\documentclass[12pt]{amsart}

\usepackage{amsmath, amsthm, amssymb, amsfonts, amscd, graphicx, endnotes, tikz, color, hyperref}

\usepackage{mathtools}

\usepackage[margin=01.4in]{geometry}

\usetikzlibrary{arrows}

\usepackage[nodisplayskipstretch]{setspace}
\def\AA{{\mathbb A}}

\def\CC{{\mathbb C}}

\def\GG{{\mathbb G}}

\def\QQ{{\mathbb Q}}

\def\QQ{{\mathbb Q}}
\def\RR{{\mathbb R}}

\def\ZZ{{\mathbb Z}}

\def\pfrak{{\mathfrak p}}
\def\qfrak{{\mathfrak q}}

\def\hhat{{\hat h}}

\def\0{{\mathbf 0}}
\def\1{{\mathbf 1}}

\def\Abf{{\mathbf A}}

\def\Mcal{{\mathcal M}}

\def\Ocal{{\mathcal O}}

\def\Kbar{{\bar K}}

\def\Aut{\mathrm{Aut}}

\def\disc{\mathrm{disc}}
\def\sgn{\mathrm{sgn}}

\def\ab{\mathrm{ab}}
\def\Gal{\mathrm{Gal}}

\def\sup{\mathrm{sup}}
\def\max{\mathrm{max}}

\theoremstyle{plain}

\newtheorem{thm}{Theorem}
\newtheorem{conj}{Conjecture}
\newtheorem{cor}[thm]{Corollary}
\newtheorem{prop}[thm]{Proposition}
\newtheorem{lem}[thm]{Lemma}

\theoremstyle{definition}

\title[Abelian extensions in dynamical Galois theory]{Abelian extensions in dynamical Galois theory}

\author{Jesse Andrews}

\address{Jesse Andrews; Department of Mathematics and Computer Science; Washington College; 300 Washington Avenue; Chestertown MD 21620 U.S.A.}

\email{jandrews4@washcoll.edu}

\author{Clayton Petsche}

\address{Clayton Petsche; Department of Mathematics; Oregon State University; Corvallis OR 97331 U.S.A.}

\email{petschec@math.oregonstate.edu}

\begin{document}

\begin{abstract}
We propose a conjectural characterization of when the dynamical Galois group associated to a polynomial is abelian, and we prove our conjecture in several cases, including the stable quadratic case over $\QQ$.  In the postcritically infinite case, the proof uses algebraic techniques, including a result concerning ramification in towers of cyclic $p$-extensions.  In the postcritically finite case, the proof uses the theory of heights together with results of Amoroso-Zannier and Amoroso-Dvornicich, as well as properties of the Arakelov-Zhang pairing.
\end{abstract}

\maketitle


\section{Introduction}

Let $K$ be a number field with algebraic closure $\Kbar$.  Let $\phi(x)\in K[x]$ be a polynomial of degree $d\geq2$, and denote by $\phi^n=\phi\circ\dots\circ\phi$ the $n$-fold composition of $\phi$ with itself.  Let $\alpha\in K$ be a non-exceptional point for $\phi$; that is, assume that the backward orbit $\{\beta\in\Kbar\mid \phi^n(\beta)=\alpha\text{ for some }n\geq0\}$ of $\alpha$ is an infinite set.

For each $n\geq1$, define the $n$-th inverse image set of the pair $(\phi,\alpha)$ by 
\begin{equation*}
\phi^{-n}(\alpha) = \{\beta\in\Kbar\mid \phi^n(\beta)=\alpha\},
\end{equation*}
and let $K_n=K_n(\phi,\alpha)$ be the field generated over $K$ by $\phi^{-n}(\alpha)$.  Since the generators of $K_n$ are $\phi$-images of generators of $K_{n+1}$, we obtain a tower $K=K_0\subseteq K_1\subseteq K_2\subseteq\dots$ of Galois extensions of $K$.  Set $K_\infty=\cup_{n\geq0}K_n$.  

As described for example in \cite{MR3220023}, $\Gal(K_n/K)$ acts faithfully on the $n$-th preimage tree $T_n=T_n(\phi,\alpha)$ associated to the pair $(\phi,\alpha)$, which can be described as follows.  For each $0\leq m\leq n$, the level-$m$ vertices of $T_n$ are indexed by the elements of $\phi^{-m}(\alpha)$, and edge relations on $T_n$ are determined by $\phi$-evaluation.  In the limit as $n\to+\infty$, $\Gal(K_\infty/K)$ acts faithfully on $T_\infty=\cup T_n$, and we obtain the arboreal Galois representations
\begin{equation}\label{ArborealRep}
\begin{split}
\rho_n & :\Gal(K_n/K)\hookrightarrow\Aut(T_n) \\
\rho & :\Gal(K_\infty/K)\hookrightarrow\Aut(T_\infty).
\end{split}
\end{equation}

The study of the representations (\ref{ArborealRep}) goes back to Odoni (\cite{MR805714}, \cite{MR813379}, \cite{MR962740}, \cite{MR1418355}) and Stoll \cite{MR1174401} in the 1980s-1990s, and has found renewed interest since the mid 2000s due to a series of papers by Boston \cite{BostonTreeRep}, Boston-Jones \cite{MR2318536}, \cite{MR2520459} and Jones \cite{jones:thesis}, \cite{jones:itgaltow}, \cite{jones:denprimdiv}, \cite{MR3220023}.  Much of the current research in this area focuses on identifying cases in which $\Gal(K_\infty/K)$ is large in the sense that the arboreal Galois representation $\rho$ is surjective, or has image with finite index in $\Aut(T_\infty)$. 
 
Assume for now that the pair $(\phi,\alpha)$ is {\em stable}, that is that the $\Gal(K_n/K)$-action on $\phi^{-n}(\alpha)$ is transitive for each $n\geq1$; this is equivalent to the irreducibility of $\phi^n(x)-\alpha$ for all $n\geq1$.  In this case, each $T_n$ is the complete $d$-ary rooted tree of level $n$, so using transitivity and comparing with the size of $\Aut(T_n)$, it follows from the injectivity of (\ref{ArborealRep}) that 
\begin{equation}\label{GaloisRange}
d^n\leq|\Gal(K_n/K)|\leq  d!^{(d^{n}-1)/(d-1)}
\end{equation}
for all $n$.  Examples in which the upper bound in (\ref{GaloisRange}) is achieved for all $n\geq1$ have been identified by Odoni \cite{MR813379} and Stoll \cite{MR1174401} in degree $d=2$, by Looper \cite{MR3937588} in every prime degree, and by Specter \cite{Specter} in arbitrary degree.

In the opposite direction, let us say that a pair $(\phi,\alpha)$ is {\em minimally stable} if it is stable, and the lower bound in (\ref{GaloisRange}) is achieved for all $n\geq1$.  For example, let $K=\QQ$, $\phi(x)=x^2$, and $\alpha=-1$.  This pair $(\phi,\alpha)$ is stable, indeed $\phi^n(x)-\alpha=x^{2^n}+1$ is the $2^{n+1}$-th cyclotomic polynomial and hence is irreducible over $\QQ$, and $|\Gal(K_n/\QQ)|=[K_n:\QQ]=2^n$.  Since $K_\infty/\QQ$ is cyclotomic, $\Gal(K_\infty/\QQ)$ is abelian.  (In fact $\Gal(K_\infty/\QQ)\simeq\ZZ_2^\times\simeq\{\pm\}\times\ZZ_2$.)

More generally, an elementary argument shows that if the pair $(\phi,\alpha)$ is stable and $\Gal(K_\infty/K)$ is abelian, then $(\phi,\alpha)$ is minimally stable; see Lemma~\ref{ActionLem}.  We do not know whether the converse is true; i.e. whether the only minimally stable pairs $(\phi,\alpha)$ are those for which $\Gal(K_\infty/K)$ is abelian.  We do not directly address this question here.  Instead, in this paper we consider the following question: for precisely which pairs $(\phi,\alpha)$ is $\Gal(K_\infty/K)$ abelian?  In the stable case, this is closely related to the question of characterizing minimally stable pairs $(\phi,\alpha)$, but the question makes sense even in the absence of a stability hypothesis.  We conjecture that in general, $\Gal(K_\infty/K)$ is abelian only in cases related to the powering map example described above, or to similar examples arising from Chebyshev polynomials.  

Given a field extension $L/K$, we say the pair $(\phi,\alpha)$ is {\em conjugate} over $L$ to the pair $(\psi,\beta)$ if there exists an affine transformation $\gamma(x)=ax+b$ defined over $L$ such that $\psi=\gamma\circ\phi\circ\gamma^{-1}$ and $\beta=\gamma(\alpha)$.  It is not hard to see that if $(\phi,\alpha)$ and $(\psi,\beta)$ are conjugate over $K$, then $K_\infty(\phi,\alpha)=K_\infty(\psi,\beta)$.  But for us, the more important fact is that whether or not $\Gal(K_\infty(\phi,\alpha)/K)$ is abelian is an invariant of the $K^\ab$-conjugacy class of the pair $(\phi,\alpha)$, where $K^\ab$ is the maximal abelian extension of $K$ in $\Kbar$; see Proposition~\ref{AbelConjProp}.

\begin{conj}\label{PolynomialConjecture}
Let $K$ be a number field, let $\phi(x)\in K[x]$ be a polynomial of degree $d\geq2$, let $\alpha\in K$, and assume that $\alpha$ is not an exceptional point for $\phi$.  Then $K_\infty(\phi,\alpha)/K$ is an abelian extension if and only if the pair $(\phi,\alpha)$ is $K^\ab$-conjugate to the pair $(\psi,\beta)$ occuring in one of the following two families of examples:
\begin{itemize}
	\item[(i)] $\psi(x)=x^{d}$ and $\beta=\zeta$, a root of unity in $\Kbar$.
	\item[(ii)] $\psi(x)=T_d(x)$ is the $d$-th Chebyshev polynomial and $\beta=\zeta+\zeta^{-1}$, where $\zeta$ is a root of unity in $\Kbar$.
\end{itemize}
\end{conj}

As a special case, when $K=\QQ$ and $d=2$, we recall the well-known fact that every quadratic polynomial over $\QQ$ is $\overline{\QQ}$-conjugate to $x^2+c$ for a unique $c\in\QQ$, and moreover any such $\overline{\QQ}$-conjugacy is actually defined over $\QQ$.  Thus Conjecture~\ref{PolynomialConjecture} asserts in this case that for a pair $(\phi,\alpha)$ defined over $\QQ$ with $\deg(\phi)=2$, the extension $K_\infty(\phi,\alpha)/\QQ$ is abelian if and only if $(\phi,\alpha)$ is $\QQ$-conjugate to $(x^2,\pm1)$ or $(x^2-2,\beta)$ for $\beta=0,\pm1,\pm2$.

We prove partial results toward Conjecture~\ref{PolynomialConjecture} which can be divided into three main categories.  First, we prove Conjecture~\ref{PolynomialConjecture} in the quadratic, stable, postcritically infinite case  (Theorem~\ref{StableNPCFCaseThm}).  (Recall that a quadratic polynomial $\phi(x)$ is said to be {\em postcritically finite} if its critical point is $\phi$-preperiodic; otherwise it is postcritically infinite.)  The main ideas in this proof are algebraic, and culminate in showing under the above hypotheses that if $K_\infty/K$ were abelian, then no primes of $K$ with odd residue characteristic would ramify in $K_\infty$, in contradiction with a result of Bridy et. al. \cite{MR3762687} on arbitrary postcritically infinite maps.

Next, we prove Conjecture~\ref{PolynomialConjecture} for polynomials $\phi$ which are $\Kbar$-conjugate to either a powering map or a Chebyshev map (Theorems~\ref{PoweringTheorem} and \ref{ChebyshevTheorem}).  These proofs use the theory of heights together with a result of Amoroso-Zannier \cite{MR1817715} (generalizing a result of Amoroso-Dvornicich \cite{MR1740514}), giving a lower bound on the heights of elements in abelian extensions of number fields.  Notably, the results on powering and Chebyshev maps do not require a stability hypothesis.

Finally, we treat the particular postcritically finite map $\phi(x)=x^2-1$.  Using a combination of the ramification techniques of Theorem~\ref{StableNPCFCaseThm} with the height techniques of Theorems~\ref{PoweringTheorem} and \ref{ChebyshevTheorem} (and in particular a lower bound on the height in certain cyclotomic extensions due to Amoroso-Dvornicich \cite{MR1740514}), we prove Conjecture~\ref{PolynomialConjecture} for stable pairs $(x^2-1,\alpha)$ over $\QQ$.  We point out that the proof of this result is computer-assisted, in the sense that the key step in the proof is to  numerically calculate the Arakelov-Zhang pairing $\langle x^2-1,x^2\rangle$ with enough precision to show that it is less than the Bogomolov constant of the maximal abelian extension of $\QQ$ unramified at all odd primes.  In particular, we use SageMath to calculate a sum of elementary approximations to local height functions evaluated at roots of unity.

Combining these results, and using the well known fact that every quadratic polynomial over $\QQ$ is either postcritically infinite or else $\QQ$-conjugate to either the squaring map $x^2$, the Chebyshev map $x^2-2$, or $x^2-1$, we obtain the following.

\begin{thm}\label{MainStabQuadThm}
Conjecture~\ref{PolynomialConjecture} is true for all quadratic stable pairs $(\phi,\alpha)$ over $\QQ$.
\end{thm}

It is well-known that any iterate of an Eisenstein polynomial in $\ZZ[x]$ is again Eisenstein, so the pair $(\phi(x),0)$ is stable whenever $\phi(x)\in\ZZ[x]$ is Eisenstein.  Using this observation, we can give the following simple examples to show that in each of the cases described above, stable pairs $(x^2+c,\alpha)$ exist over $\QQ$ and hence Theorem~\ref{MainStabQuadThm} is non-vacuous in each case.  
\begin{itemize}
\item[(i)]  For any prime $p$, the (post-critically infinite) pair $(x^2+p,0)$ is stable.  
\item[(ii)]  If $\alpha\in\ZZ$ and $\alpha\equiv2\text{ or }3\pmod{4}$, then the squaring pair $(x^2,\alpha)$ is stable, since it is conjugate to $(x^2+2\alpha x+\alpha^2-\alpha,0)$, which is $2$-Eisenstein.  Note that this family includes both abelian examples, such as $(x^2,-1)$, and nonabelian examples, such as $(x^2,3)$.
\item[(iii)]  If $\alpha\in\ZZ$ and $\alpha\equiv0\text{ or }1\pmod{4}$, then the Chebyshev pair $(x^2-2,\alpha)$ is stable, since it is conjugate to $(x^2+2\alpha x+\alpha^2-\alpha-2,0)$, which is $2$-Eisenstein.  Note that this family includes both abelian examples, such as $(x^2-2,0)$, and nonabelian examples, such as $(x^2-2,4)$.
\item[(iv)]  This example was shown to us by Chifan Leung.  If $\alpha\in\ZZ$ and $\alpha\equiv1\text{ or }2\pmod{4}$, then the pair $(x^2-1,\alpha)$ is stable.  It suffices to show that $(\phi^2,\alpha)$ is stable, where $\phi(x)=x^2-1$ and $\phi^2(x)=x^4-2x^2$, since the irreducibility of $\phi^{2n}(x)-\alpha$ implies the irreducibility of $\phi^{2n-1}(x)-\alpha$.  The stability of $(\phi^2,\alpha)$ follows from the fact that it is conjugate to $(\phi^2(x+\alpha)-\alpha,0)$, which is easily checked to be $2$-Eisenstein.  (See also \cite{ABCCF} for a study of large-image results for arboreal Galois representations associated to $\phi(x)=x^2-1$.)
\end{itemize}

While this paper was under review, A. Ferraguti and C. Pagano \cite{FerragutiPagano} have informed us that they have used an entirely different approach to give a complete proof of the $K=\QQ$, $d=2$ case of Conjeture~\ref{PolynomialConjecture} (not requiring any stability assumption).

The plan of this paper is as follows.  In $\S$~\ref{AlgLemSect} we prove some preliminary algebraic lemmas, and in $\S$~\ref{NonPCFSect} we prove Conjecture~\ref{PolynomialConjecture} in the quadratic, stable, postcritically infinite case.  In $\S$~\ref{HeightsSect} we review the absolute Weil height function defined on algebraic extensions of $\QQ$, we recall the concept of the Bogomolov constant associated to such fields, and we describe related results of Amoroso-Zannier \cite{MR1817715} and Amoroso-Dvornicich \cite{MR1740514}.  In $\S$~\ref{PowerChebSect} we prove Conjecture~\ref{PolynomialConjecture} for powering maps and Chebyshev maps.  In $\S$~\ref{AZPairingSect} we review the definition and basic facts about the Arakelov-Zhang pairing, and in $\S$~\ref{x2MinusOneSect} and $\S$~\ref{ApproxSect} we treat the particular polynomial $\phi(x)=x^2-1$, calculate the Arakelov-Zhang pairing $\langle x^2-1,x^2\rangle$, and prove Conjecture~\ref{PolynomialConjecture} for stable pairs $(x^2-1,\alpha)$ over $\QQ$.  

\medskip

Acknowledgements: We thank Rafe Jones for several helpful suggestions.

\section{Some algebraic lemmas}\label{AlgLemSect}

\begin{lem}\label{ActionLem}
Let $G$ be a finite abelian group acting faithfully and transitively on a finite set $X$.  Then $|G|=|X|$.
\end{lem}
\begin{proof}
For each $x\in X$, let $G_x$ be the stabilizer of $x$.  Then $G_x=G_{y}$ for all $x,y\in X$.  Indeed, writing $y=gx$ for $g\in G$, if $h\in G_x$ then $hy=hgx=ghx=gx=y$, showing that $h\in G_{y}$ as well.  Thus $G_x\subseteq G_{y}$, and $G_x= G_{y}$ follows from symmetry.  Since the action is faithful, we have $\cap_{x\in X}G_x=\{1\}$, and since the stabilizers are all equal to each other we conclude that $G_{x}=\{1\}$ for all $x\in X$. Therefore $|X|=(G:G_{x})=(G:1)=|G|$ by the orbit stabilizer theorem.
\end{proof}

\begin{lem}\label{AbelianSnLem}
If $G$ is an abelian, transitive subgroup of $S_{N}$ and if $\sigma\in G$ is an element of order $\ell$, then $\sigma=c_1c_2\dots c_r$ for some $r$ disjoint $\ell$-cycles $c_1,c_2,\dots, c_r$, where $r=N/\ell$.
\end{lem}
\begin{proof}
Recall the standard calculation that if $(i_1\dots i_\ell)\in S_N$ is a cycle and if $\tau\in S_N$, then $\tau (i_1\dots i_\ell)\tau^{-1}=(\tau(i_1)\dots\tau(i_\ell))$.  

We may write $\sigma=c_1c_2\dots c_r$ for some $r$ disjoint cycles $c_1,c_2,\dots, c_r$ of lengths $\ell_1,\ell_2,\dots, \ell_r$, respectively, and this decomposition is unique up to ordering.  If necessary, interpreting some of the cycles $c_j$ to be $1$-cycles, we may assume that every element of $\{1,2,\dots,N\}$ occurs in precisely one of the cycles $c_j$.

Fix $2\leq j\leq r$.  By transitivity, select $\tau\in G$ taking some element of $\{1,2,\dots,N\}$ occurring in the cycle $c_j$ to some element occurring in the cycle $c_1$.  Since $G$ is abelian,
$$
\sigma = \tau\sigma\tau^{-1} = (\tau c_1\tau^{-1})\dots(\tau c_r\tau^{-1}),
$$
and so by uniqueness of the disjoint cycle decomposition of $\sigma$, we conclude that $c_1=\tau c_j\tau^{-1}$.  In particular, all of the cycles $c_j$ have the same length $\ell_1$, which must then be equal to the order $\ell$ of $\sigma$.  Finally, $r\ell=N$, as every element of $\{1,2,\dots,N\}$ occurs in precisely one of the cycles $c_j$.
\end{proof}

\begin{lem}\label{CyclicSnLem}
Let $G$ be an abelian, transitive subgroup of $S_{2^n}$ which is not a subgroup of $A_{2^n}$.  Then $G$ is cyclic.
\end{lem}
\begin{proof}
By Lemma~\ref{ActionLem}, we have $|G|=2^n$.  Let $\sigma\in G$ be an odd permutation of order $\ell$; thus $\ell\geq2$ is a power of $2$.  By Lemma~\ref{AbelianSnLem}, we have a decomposition $\sigma=c_1c_2\dots c_r$ into disjoint $\ell$-cycles $c_j$, and $r\ell=2^n$.  Since $\ell$ is even, $\sgn(c_j)=-1$ for all $j$, and therefore 
$$
-1=\sgn(\sigma) = \prod_{1\leq j\leq r}\sgn(c_j) = (-1)^{r}.
$$
Thus $r$ is odd, and as $r\ell=2^n$, we must have $r=1$.  We conclude that $\sigma=c_1$ is a $2^n$-cycle and hence that $G=\langle\sigma\rangle$ is cyclic.
\end{proof}

The assumption that $G\not\subseteq A_{2^n}$ cannot be omitted.  For example, the order $8$ subgroup $G=\langle\sigma,\tau\rangle$ of $A_8$ generated by the (commuting) permutations $\sigma=(1537)(2648)$ and $\tau=(12)(34)(56)(78)$ is abelian and transitive, but not cyclic.  We also point out that this counterexample cannot be removed using properties of tree automorphisms, as we may view $G$ as a subgroup of the automorphism group of a binary rooted tree of level $3$, by embedding the tree in the usual way in the plane and labeling the level-$3$ vertices by the numbers $1,\dots,8$ from left to right.

\begin{lem}\label{DiscIterateLem}
Let $f(x)=Ax^2+Bx+C\in K[x]$  be a quadratic polynomial, and let $c=-B/2A$ be its critical point.  Then for all $n\geq1$, 
\begin{equation}\label{DiscIterateIdentity}
\disc(f^n) = (-1)^{2^{n-1}}2^{2^n}A^{2^{2n-1}-1}\disc(f^{n-1})^2f^n(c).
\end{equation}
\end{lem}

This identity is worked out (in greater generality) up to sign by Jones in \cite{jones:denprimdiv} Lemma 2.6; it is straightforward to go through Jones' calculation and keep track of the factor $(-1)^{2^{n-1}}$, which of course is $-1$ when $n=1$ and $+1$ when $n\geq2$.  To check (\ref{DiscIterateIdentity}) when $n=1$, take $f^0(x)=x$ and hence $\disc(f^0)=1$, which is reasonable as one typically interprets the empty product to be $1$.  In this case, the right hand side of (\ref{DiscIterateIdentity}) simplifies to $-4Af(c)=B^2-4AC$, as expected.

\section{Ramification and postcritically infinite quadratic maps}\label{NonPCFSect}

We recall standard facts and notation surrounding a finite Galois extension $L/K$ of number fields; see Lang \cite{lang:numbertheory} Ch. 1.  Given a prime $\pfrak$ of $K$, by the Galois assumption we have a factorization of the form $\pfrak\Ocal_L=\qfrak_1^e\dots\qfrak_r^e$ for primes $\qfrak_1,\dots\qfrak_r$ of $L$, and $ref=[L:K]$, where $e=e(\qfrak_i/\pfrak)$ and $f=f(\qfrak_i/\pfrak)$ are the (common) ramification indices and inertial degrees of the  $\qfrak_i$, respectively.  Moreover, each $\Ocal_L/\qfrak_i$ is a degree $f$ extension of $\Ocal_K/\pfrak$.  For each $1\leq i\leq r$, let  
\begin{equation*}
\begin{split}
D_{\qfrak_i/\pfrak} & = \{\sigma\in\Gal(L/K)\mid \sigma(\qfrak_i)=\qfrak_i\} \\
I_{\qfrak_i/\pfrak} & = \{\sigma\in D_{\qfrak_i/\pfrak}\mid \sigma(x)\equiv x \pmod{\qfrak_i} \text{ for all }x\in\Ocal_L\}
\end{split}
\end{equation*}
be the associated decomposition and inertia groups.  Thus $I_{\qfrak_i/\pfrak}$ has order $e(\qfrak_i/\pfrak)$, $\pfrak$ is unramified in $L^{I_{\qfrak_i/\pfrak}}$, and if $\pfrak'$ denotes any prime of $L^{I_{\qfrak_i/\pfrak}}$ lying over $\pfrak$, then $\pfrak'$ is totally ramified in $L$.

\begin{lem}\label{RamBoundLem}
Let $L/K$ be a Galois extension of number fields and let $\pfrak$ be a prime of $K$ which is tamely ramified in $L$.  Let $\qfrak$ be a prime of $L$ lying over $\pfrak$.  Then 
$$
e(\qfrak/\pfrak)\leq |\Ocal_K/\pfrak|^{f(\qfrak/\pfrak)}-1.
$$
\end{lem}
\begin{proof}
Let $\pi\in\Ocal_L$ be a uniformizer for $\qfrak$, and consider the group homomorphism
\begin{equation*}
\begin{split}
I_{\qfrak/\pfrak} & \to(\Ocal_L/\qfrak)^\times \\
	\sigma & \mapsto \sigma(\pi)/\pi \pmod{\qfrak}
\end{split}
\end{equation*}
Standard arguments from the theory of local fields show that this map does not depend on the choice of uniformizer, and the tame ramification hypothesis implies that it is injective; see \cite{MR911121} $\S$I.8.  Together with the fact that $|\Ocal_L/\qfrak|=|\Ocal_K/\pfrak|^{f(\qfrak/\pfrak)}$, we obtain the desired inequality.
\end{proof}

\begin{lem}\label{CyclicTowerRamLem}
Let $K$ be a number field, let $K=K_0\subset K_1\subset K_2\subset\dots$ be a tower of distinct cyclic $p$-extensions of $K$, and let $K_\infty=\cup K_n$.  If $\pfrak$ is a prime of $K$ with residue characteristic not equal to $p$, then $\pfrak$ is unramified in $K_\infty$.
\end{lem}
\begin{proof}
Since a quotient of a cyclic $p$-group is another cyclic $p$-group, without loss of generality we may insert intermediate fields and reindex to ensure that $[K_{n}:K_{n-1}]=p$ for all $n\geq1$.  Since $\Gal(K_n/K)$ is a cyclic $p$-group, its subgroups are totally ordered by inclusion, and thus the same is true of intermediate fields $K\subseteq F\subseteq K_n$. In particular, the fields $K=K_0\subset K_1\subset\dots\subset K_n$ are the only subfields of $K_n$ containing $K$.

Contrary to what has been claimed, assume that $\pfrak$ has residue characteristic not equal to $p$ and that $\pfrak$ ramifies (hence tamely ramifies) in $K_\infty$.  Let $\pfrak_0=\pfrak$, and for each $n\geq 1$, let $\pfrak_n$ be a prime of $K_n$ lying over $\pfrak_{n-1}$.  Let $n_0$ be maximal with the property that $\pfrak$ is unramified in $K_{n_0}$; thus $\pfrak_{n_0}$ is ramified in $K_{n_0+1}$.  Let $n>n_0$ be arbitrary, and define $F_\pfrak=K_n^{I_{\pfrak_n/\pfrak}}$, the fixed field of the inertia subgroup $I_{\pfrak_n/\pfrak}$ of $\Gal(K_n/K)$.  In particular, $\pfrak$ is unramified in $F_\pfrak$, and if $\pfrak'$ denotes any prime of $F_\pfrak$ lying over $\pfrak$, then $\pfrak'$ is totally ramified in $K_n$.  Since we must have $F_\pfrak=K_{m}$ for some $0\leq m\leq n$, the only possibility is $F_\pfrak=K_{n_0}$.  

To summarize, we have shown that $\pfrak$ is unramified in $K_{n_0}$, and that $\pfrak_{n_0}$ is totally ramified in $K_n$ for all $n>n_0$.  In particular, we have
\begin{equation}\label{eAndfBounds}
\begin{split}
f(\pfrak_n/\pfrak) & = f(\pfrak_{n_0}/\pfrak) \leq [K_{n_0}:K] = p^{n_0} \\
e(\pfrak_n/\pfrak) & = e(\pfrak_{n}/\pfrak_{n_0}) = [K_{n}:K_{n_0}] = p^{n-n_0}. 
\end{split}
\end{equation}
But for large enough $n$, (\ref{eAndfBounds}) contradicts the bound
$$
e(\pfrak_n/\pfrak)\leq |\Ocal_K/\pfrak|^{f(\pfrak_n/\pfrak)}-1.
$$
which follows from Lemma~\ref{RamBoundLem}.
\end{proof}

\begin{thm}\label{StableNPCFCaseThm}
Let $\phi(x)\in K[x]$ be a quadratic polynomial which is not postcritically finite, let $\alpha \in K$, and assume that the pair $(\phi,\alpha)$ is stable.  Then $\Gal(K_\infty/K)$ is nonabelian.
\end{thm}

\begin{proof}
Let $\phi(x)\in K[x]$ be a quadratic polynomial which is not postcritically finite, let $\alpha \in K$, assume that the pair $(\phi,\alpha)$ is stable, and assume that $\Gal(K_\infty/K)$ is abelian; we will obtain a contradiction.  

We first prove that $\Gal(K_n/K)$ is cyclic for all $n\geq1$.  To see this, note first that the stability and abelian hypotheses imply via Lemma~\ref{ActionLem} that $[K_n:K]=2^n$ for all $n\geq1$.  It suffices to show that $\Gal(K_n/K)$ is cyclic for arbitrarily large $n$, because if $\Gal(K_n/K)$ is cyclic then so are its quotients $\Gal(K_{m}/K)$ for $1\leq m<n$.  By the stability hypothesis and Lemma~\ref{CyclicSnLem}, it suffices to show, for arbitrarily large $n$, that $\Gal(K_n/K)$ is not contained in $A_{2^n}$ when viewed as a subgroup of $S_{2^n}$ via its action on the roots of $\phi^n(x)-\alpha$.  Suppose on the contrary that $\Gal(K_n/K)\subseteq A_{2^n}$ for all sufficiently large $n$.  By a well-known exercise in elementary Galois theory, this means that $\disc(\phi^n(x)-\alpha)$ is a square in $K$ for all sufficiently large $n$.  Letting $\psi(x)=\phi(x+\alpha)-\alpha$, using Lemma~\ref{DiscIterateLem} we have
$$
\disc(\phi^n(x)-\alpha) = \disc(\phi^n(x+\alpha)-\alpha)=\disc(\psi^n(x))=R_n^2A\psi^n(c)
$$
for all $n\geq2$, where $A,R_n\in K$ are nonzero and where $c$ is the critical point of $\psi(x)$.  In particular, $A\psi^n(c)$ is a square in $K$ for all sufficiently large $n$.  

The pair $(\psi,0)$ is stable by the stability assumption on the pair $(\phi,\alpha)$.  In particular, the degree $8$ polynomial $\psi^3(x)$ has eight distinct roots in $\Kbar$, and thus $C=\{y^2=A\psi^3(x)\}$ is a smooth hyperelliptic curve of genus $3$.  There are infinitely many $n\geq3$ for which $A\psi^n(c)$ is a square in $K$ and hence for which $\psi^{n-3}(c)$ is the $x$-coordinate of a $K$-rational point on $C$.  Moreover, these points are distinct by the postcritically infinite hypothesis on $\phi$ (and hence on $\psi$ as well).  This violation of Faltings theorem provides a contradiction, and thus the assumption $\Gal(K_n/K)\subseteq A_{2^n}$ for all large enough $n$ is false.  As explained above, by Lemma~\ref{CyclicSnLem} this completes the proof that $\Gal(K_n/K)$ is cyclic for all $n\geq1$.  

We now apply the $p=2$ case of Lemma~\ref{CyclicTowerRamLem}, which implies that no primes $\pfrak$ of $K$ with odd residue characteristic can ramify in $K_\infty$.  However, this violates a theorem of Bridy et. al. \cite{MR3762687}, which states that if $K_\infty$ is generated over $K$ by the preimage tree associated to a postcritically infinite rational map, then infinitely many primes of $K$ ramify in $K_\infty$.  The contradiction completes the proof of the theorem.
\end{proof}

The use of Falting's theorem to limit the number of squares in the critical orbit of a polynomial is borrowed from Boston-Jones \cite{MR2520459}.  In fact, Theorem~\ref{StableNPCFCaseThm} may be viewed as a generalization of Theorem 3.1 of \cite{MR2520459}, in the sense that our result implies that the hypotheses of that theorem can never be satisfied.

\section{Heights and Bogomolov constants}\label{HeightsSect}

We recall the definition of the absolute Weil height function $h:\Kbar\to\RR$ for a number field $K$.  For each finite extension $L/K$, denote by $M_L$ the set of places of $L$, and for each place $v$ let $|\cdot|_v$ be a corresponding absolute value normalized so that it coincides with either the standard real or $p$-adic absolute value when restricted to $\QQ$.  Given $\alpha\in\Kbar$, Let $L/K$ be a finite extension containing $\alpha$, and define 
\begin{equation}\label{WeilHeightDef}
h(\alpha) = \sum_{v\in M_L}r_v\log^+|\alpha|_v
\end{equation}
where $r_v=[L_v:\QQ_v]/[L:\QQ]$ and $\log^+t=\log\max(1,t)$.  Standard arguments show that this definition is independend of the choice of $L$, and that $h(\alpha)\geq0$ for all $\alpha\in \Kbar$, with $h(\alpha)>0$ unless $\alpha$ is zero or a root of unity.  It is immediate from the definition that $h(\zeta\alpha)=h(\alpha)$ for all roots of unity $\zeta$, and that $h(\alpha^n)=|n|h(\alpha)$ for all $n\in\ZZ$.

Given a field $K\subseteq L\subseteq\Kbar$ (with $L/K$ not necessarily a finite extension), define the Bogomolov constant of $L$ by
$$
B_0(L) = \liminf\{h(\alpha)\mid \alpha\in L\text{ and }h(\alpha)>0\}.
$$
In other words, $B_0(L)$ is the unique extended real number $[0,+\infty]$ with the property that the set $\{\alpha\in L \mid 0<h(\alpha)\leq B\}$ is finite for all $B<B_0(L)$ and infinite for all $B>B_0(L)$.  
 
\begin{thm}[Amoroso-Zannier \cite{MR1817715}]\label{AmorosoZannierBound}
If $L/K^\ab$ is a finite extension of degree $D=[L:K^\ab]$, then $h(\alpha)\geq C_{K,D}>0$ for all nonzero, non-root of unity $\alpha\in L$, where $C_{K,D}$ is a constant depending only on $K$ and $D$.  In particular, $B_0(L)\geq C_{K,D}>0$.
\end{thm}

This result generalizes a result of Amoroso-Dvornicich \cite{MR1740514}, which states that $h(\alpha)\geq(\log5)/12$ for all nonzero, non-root of unity $\alpha\in \QQ^\ab$.  In particular, $B_0(\QQ^\ab)\geq (\log5)/12$.  For our purposes, another useful result from the paper \cite{MR1740514} is the following.  For each $k\geq1$, let $\zeta_k$ be a primitive $k$-th root of unity in $\CC$, and let $\mu_k$ be the group of all $k$-th roots of unity in $\CC$.  Let $\mu_{2^\infty}=\cup_{m\geq1}\mu_{2^m}$; thus $ \QQ(\mu_{2^\infty})=\cup_{m\geq1}\QQ(\zeta_{2^m})$.  

\begin{thm}[Amoroso-Dvornicich \cite{MR1740514}]\label{AD2Theorem}
If $\alpha\in  \QQ(\mu_{2^\infty})$ is nonzero and not a root of unity, then 
$h(\alpha)\geq(\log2)/4$.  In particular, $B_0( \QQ(\mu_{2^\infty}))\geq (\log2)/4$.
\end{thm}

Basically all of the ideas needed to prove this result are present in Proposition 2 of \cite{MR1740514}, which treats the cyclotomic fields $\QQ(\zeta_k)$ for $4\mid k$.  The statement of the height bound in that result excludes certain elements of $\QQ(\zeta_k)$, but we can easily recover the bound for these excluded elements in the special case that $k=2^m$.  As it may be of some interest, we include the complete proof in this case.

\begin{proof}[Proof of Theorem~\ref{AD2Theorem}]
If $\zeta\alpha\in\QQ$ for some root of unity $\zeta\in \QQ(\mu_{2^\infty})$, then $\zeta\alpha\notin\{0,\pm1\}$ and so $h(\alpha)=h(\zeta\alpha)\geq\log2>\frac{\log2}{4}$.  Thus we may assume that $\zeta\alpha\notin\QQ$ for all roots of unity $\zeta\in \QQ(\mu_{2^\infty})$.  Let $m$ be the smallest positive integer with the property that $\zeta\alpha\in\QQ(\zeta_{2^m})$ for some root of unity $\zeta\in \QQ(\mu_{2^\infty})$; thus $m\geq2$ by assumption.  Since $h(\zeta\alpha)=h(\alpha)$, without loss of generality we may just assume that $\alpha\in \QQ(\zeta_{2^m})$ and that $\zeta\alpha\notin\QQ(\zeta_{2^{m-1}})$ for all roots of unity $\zeta\in \QQ(\mu_{2^\infty})$.

Write $\Gal(\QQ(\zeta_{2^m})/\QQ(\zeta_{2^{m-1}}))=\{1,\sigma\}$; thus $\sigma(\zeta_{2^m})=-\zeta_{2^m}$.  Set 
$$
\gamma=\sigma(\alpha)^2-\alpha^2.
$$
Note that $\gamma\neq0$ as otherwise either $\sigma(\alpha)=\alpha$ or $\sigma(\alpha)=-\alpha$; the former case implies $\alpha\in \QQ(\zeta_{2^{m-1}})$, and the latter case implies $\zeta_{2^m}\alpha\in \QQ(\zeta_{2^{m-1}})$, both of which are forbidden by assumption.

If $v$ is a place of $\QQ(\zeta_{2^m})$, then 
\begin{align}
\label{GammaPlaceBound1} |\gamma|_v & \leq \max(1,|\alpha|_v)^2\max(1,|\sigma(\alpha)|_v)^2 & & \text{ if } v\nmid2,\infty \\
\label{GammaPlaceBound2} |\gamma|_v & \leq (1/4)\max(1,|\alpha|_v)^2\max(1,|\sigma(\alpha)|_v)^2 & & \text{ if } v\mid2 \\
\label{GammaPlaceBound3}|\gamma|_v & \leq 2\,\max(1,|\alpha|_v)^2\max(1,|\sigma(\alpha)|_v)^2 & & \text{ if } v\mid\infty 
\end{align}
These inequalities and the product formula, together with the fact that $h(\sigma(\alpha))=h(\alpha)$, imply that $0=\sum_vr_v\log|\gamma|_v\leq4h(\alpha)-\log4+\log2$, and the desired bound $h(\alpha)\geq(\log2)/4$ follows.  The bounds (\ref{GammaPlaceBound1}) and (\ref{GammaPlaceBound3}) are trivial applications of the triangle inequality.

It remains only to prove (\ref{GammaPlaceBound2}); thus fix a place $v\mid2$ of $\QQ(\zeta_{2^m})$.  Using Proposition Lemma 4.4.12 of \cite{bombierigubler}, there exists $\beta\in\ZZ[\zeta_{2^m}]$ such that $\alpha\beta\in\ZZ[\zeta_{2^m}]$ and $|\beta|_v=\max(1,|\alpha|_v)^{-1}$.  Note that for arbitrary $x\in\ZZ[\zeta_{2^m}]$, writing $x=\Sigma_ja_j\zeta_{2^m}^j$, since $\sigma(\zeta_{2^m})=-\zeta_{2^m}$ we have 
$$
\sigma(x)^2-x^2=(\sigma(x)-x)(\sigma(x)+x)=-4(\sum_{2\nmid j}a_j\zeta_{2^m}^j)(\sum_{2\mid j}a_j\zeta_{2^m}^j)
$$
and thus $|\sigma(x)^2-x^2|_v\leq1/4$.  We conclude
\begin{equation*}
\begin{split}
|\beta|_v^2|\gamma|_v & = |\beta^2\sigma(\alpha)^2-\alpha^2\beta^2|_v \\
	& = |(\beta^2-\sigma(\beta)^2)\sigma(\alpha)^2+\sigma(\alpha\beta)^2-(\alpha\beta)^2|_v \\
	& \leq \max\big(|\beta^2-\sigma(\beta)^2|_v|\sigma(\alpha)|_v^2,|\sigma(\alpha\beta)^2-(\alpha\beta)^2|_v\big) \\
	& \leq \max\big((1/4)|\sigma(\alpha)|_v^2,1/4\big) \\
	& =(1/4)\max(1,|\sigma(\alpha)|_v)^2,
\end{split}
\end{equation*}
which is equivalent to (\ref{GammaPlaceBound2}) as $|\beta|_v=\max(1,|\alpha|_v)^{-1}$.
\end{proof}

\section{Powering maps and Chebyshev maps}\label{PowerChebSect}

In a slightly more general framework than what has been described above, in this section we consider pairs $(\phi,\alpha)$, where $\phi(x)\in \Kbar[x]$ is a polynomial and where $\alpha\in\Kbar$.  We define recursively $K_0=K_0(\phi,\alpha)=K(\alpha)$ and $K_n=K_n(\phi,\alpha)=K_{n-1}(\phi^{-n}(\alpha))$ for $n\geq1$, and set $K_\infty(\phi,\alpha)=\cup K_n(\phi,\alpha)$.  Since the requirement that $\phi$ and $\alpha$ are defined over $K$ have been relaxed, $K_0/K$ may be a proper extension and the $K_n/K$ may no longer be Galois extensions.  
 
\begin{prop}\label{AbelConjProp}
Let $K$ be a number field, let $\phi(x),\psi(x)\in \Kbar[x]$ be two polynomials of degree $d\geq2$, and let $\alpha,\beta\in \Kbar$.  
\begin{itemize}
\item[{\bf (a)}] If $(\phi,\alpha)$ is $\Kbar$-conjugate to $(\psi,\beta)$, then $K_\infty(\phi,\alpha)$ is contained in a finite extension of $K^\ab$ if and only if $K_\infty(\psi,\beta)$ is contained in a finite extension of $K^\ab$.
\item[{\bf (b)}] If $\phi(x)$, $\psi(x)$, $\alpha$, $\beta$ are defined over $K^\ab$  and $(\phi,\alpha)$ is $K^\ab$-conjugate to $(\psi,\beta)$, then $K_\infty(\phi,\alpha)/K$ is an abelian extension if and only if $K_\infty(\psi,\beta)/K$ is an abelian extension.
\end{itemize}
\end{prop}
\begin{proof}
{\bf (a)}  There exists a finite extension $F/K$ such that $\phi(x)$, $\psi(x)$, $\alpha$, $\beta$ are all defined over $F$, and extending $F$ if necessary there exists an automorphism $\gamma(x)=ax+b$ defined over $F$ for which $\psi=\gamma\circ\phi\circ\gamma^{-1}$ and $\beta=\gamma(\alpha)$.  Note that for each $n\geq0$, $\gamma$ restricts to a bijection from $\phi^{-n}(\alpha)$ onto $\psi^{-n}(\beta)$.  In particular, it follows that $K_\infty(\psi,\beta)\subseteq FK_\infty(\phi,\alpha)$, and thus if $K_\infty(\phi,\alpha)$ is contained in a finite extension $L$ of $K^\ab$, then $K_\infty(\psi,\beta)$ is contained in the finite extension $LF$ of $K^\ab$.  The reverse implication follows from symmetry.

{\bf (b)}  In the preceding argument, we may take $F\subseteq K^\ab$.  Thus if $K_\infty(\phi,\alpha)\subseteq K^\ab$, then $K_\infty(\psi,\beta)\subseteq K^\ab$ as well, and conversely by symmetry.
\end{proof}

The following two results verify Conjecture~\ref{PolynomialConjecture} in the special case that $\phi(x)$ is $\Kbar$-conjugate to a powering map $x^d$ or to a Chebyshev map $T_d(x)$.

\begin{thm}\label{PoweringTheorem}
Let $\phi(x)\in \Kbar[x]$ be a polynomial of degree $d\geq2$, let $\alpha\in \Kbar$ be a non-exceptional point for $\phi$, and assume that the pair $(\phi,\alpha)$ is $\Kbar$-conjugate to the pair $(x^d,\beta)$ for $\beta\in \Kbar$.  Then $K_\infty(\phi,\alpha)/K$ is an abelian extension if and only if $\beta$ is a root of unity and $(\phi,\alpha)$ is $K^\ab$-conjugate to $(x^d,\beta)$.
\end{thm}

\begin{proof}
Assume that $\beta$ is a root of unity and that $(\phi,\alpha)$ is $K^\ab$-conjugate to $(x^d,\beta)$.  Then $K_\infty(x^d,\beta)$ is a cyclotomic, and hence abelian, extension of $K$, and it follows from Proposition~\ref{AbelConjProp} {\bf (b)} that $K_\infty(\phi,\alpha)/K$ is an abelian extension.

Conversely, assume that $K_\infty(\phi,\alpha)/K$ is an abelian extension.  Using Proposition~\ref{AbelConjProp} {\bf (a)}, it follows that $K_\infty(x^d,\beta)$ is contained in a finite extension $L$ of $K^\ab$.  If $\beta$ is not a root of unity, then $h(\beta)>0$. (Note that $\beta\neq0$ by the assumption that $\alpha$ is not an exceptional point of $\phi$, and hence $\beta$ is not an exceptional point of $x^d$.)  But $\beta^{1/d^n}\in K_\infty(x^d,\beta)\subseteq L$ for all $n\geq0$, and $h(\beta^{1/d^n})=\frac{1}{d^n}h(\beta)\to0^+$ as $n\to+\infty$, a contradiction of Theorem \ref{AmorosoZannierBound}.  We conclude that $\beta$ must be a root of unity.

Finally, we must show that the $\Kbar$-conjugacy between $(\phi,\alpha)$ and $(x^d,\beta)$ is actually defined over $K^\ab$.  By hypothesis there exists $\gamma(x)=ax+b$ defined over $\Kbar$  for which $x^d=\gamma\circ\phi\circ\gamma^{-1}(x)$ and $\beta=\gamma(\alpha)$.  Moreover, $\gamma$ restricts to a bijection from the backward $\phi$-orbit of $\alpha$ onto the backward $x^d$-orbit of $\beta$.  These are infinite sets contained in $K^\ab$, since both $K_\infty(\phi,\alpha)/K$ and $K_\infty(x^d,\beta)/K$ are abelian extensions.  Selecting distinct corresponding pairs $\gamma(s_1)=t_1$ and $\gamma(s_2)=t_2$ with $s_j,t_j\in K^\ab$, we have that both $a=(t_1-t_2)/(s_1-s_2)$ and $b=(s_1t_2-t_1s_2)/(s_1-s_2)$ are in $K^\ab$.
\end{proof}

Let $d\geq2$ be an integer, and let $T_d(x)\in\ZZ[x]$ be the $d$-th Chebyshev polynomial; that is, $T_d(x)$ is the unique polynomial of degree $d$ satisfying $T_d(x+\frac{1}{x})=x^d+\frac{1}{x^d}$.  In other words, considering the $2$-to-$1$ rational map $\pi:\GG_m\to\AA^1$ defined by $\pi(x)=x+\frac{1}{x}$, we have a commutative diagram
\begin{equation}\label{ChebyshevCD}
\begin{CD}
\GG_m  @> x^d >>   \GG_m \\ 
@V \pi VV                                    @VV \pi V \\ 
\AA^1           @> T_d >>      \AA^1.
\end{CD} 
\end{equation}
See Silverman \cite{MR2316407} $\S$6.2.

\begin{thm}\label{ChebyshevTheorem}
Let $\phi(x)\in \Kbar[x]$ be a polynomial of degree $d\geq2$, let $\alpha\in \Kbar$ be a non-exceptional point for $\phi$, and assume that the pair $(\phi,\alpha)$ is $\Kbar$-conjugate to the pair $(T_d,\beta)$ for $\beta\in \Kbar$.  Then $K_\infty(\phi,\alpha)/K$ is an abelian extension if and only if $\beta=\zeta+\frac{1}{\zeta}$ for some root of unity $\zeta\in\Kbar$ and $(\phi,\alpha)$ is $K^\ab$-conjugate to $(T_d,\beta)$.
\end{thm}

\begin{proof}
Assume that $\beta=\zeta+\frac{1}{\zeta}$ for some root of unity $\zeta\in\Kbar$ and that $(\phi,\alpha)$ is $K^\ab$-conjugate to $(T_d,\beta)$.  By the commutative diagram (\ref{ChebyshevCD}), the points $\epsilon\in \Kbar$ satisfying $T_d^n(\epsilon)=\beta$ are precisely the points of the form $\epsilon=\xi+\frac{1}{\xi}$, as $\xi$ ranges over the $d^n$-th roots of $\zeta$.  In particular, $K_\infty(T_d,\beta)$ is contained in a cyclotomic, and hence abelian, extension of $K$, and it follows from Proposition~\ref{AbelConjProp} {\bf (b)} that $K_\infty(\phi,\alpha)/K$ is an abelian extension.

Conversely, assume that $K_\infty(\phi,\alpha)/K$ is an abelian extension.  Using Proposition~\ref{AbelConjProp} {\bf (a)}, it follows that $K_\infty(T_d,\beta)$ is contained in a finite extension $L$ of $K^\ab$; let $D=[L:K^\ab]$.  Select $\zeta\in\pi^{-1}(\beta)$, thus $\beta=\zeta+\frac{1}{\zeta}$, and assume that $\zeta$ is not a root of unity.  In particular $h(\zeta)>0$.  Let $n\geq0$ and select $\epsilon_n\in \Kbar$ satisfying $T_d^n(\epsilon_n)=\beta$; thus $\epsilon_n=\xi_n+\frac{1}{\xi_n}$ for some $d^n$-th root $\xi_n=\zeta^{1/d^n}$ of $\zeta$.  Since $\epsilon_n\in K_\infty(T_d,\beta)\subseteq L$, it follows that $\xi_n$ is contained in a quadratic extension of $L$ and hence contained in an extension of $K^\ab$ of degree $\leq 2D$.  It follows from Theorem \ref{AmorosoZannierBound} that $h(\xi_n)\geq C_{K,2D}$.  But as $n\geq0$ is arbitrary, we may let $n\to+\infty$ and obtain $h(\xi_n)=\frac{1}{d^n}h(\zeta)\to0^+$, a contradiction.  We conclude that $\zeta$ must be a root of unity.  That the $\Kbar$-conjugacy between $(\phi,\alpha)$ and $(T_d,\beta)$ is actually defined over $K^\ab$ follows from the same argument used in Theorem~\ref{PoweringTheorem}. 
\end{proof}

\section{Maps with small Arakelov-Zhang pairing}\label{AZPairingSect}

We now describe how to extend the ideas used in the proof of Theorem~\ref{PoweringTheorem} to treat polynomials which are not necessarily $\Kbar$-conjugate to powering maps, but which are $K^\ab$-conjugate to some polynomial $\phi(x)\in K[x]$ that is arithmetically close to a powering map.

We first recall the definitions of several arithmetic-dynamical objects associated to a polynomial $\phi(x)\in K[x]$ of degree $d\geq2$ defined over a number field $K$; see \cite{MR2316407} $\S$3.4-3.5 for further details.  The {\em Call-Silverman canonical height function} $\hhat_{\phi}:\Kbar\to\RR$ may be defined by the limit
$$
\hhat_{\phi}(x) = \lim_{n\to+\infty}\frac{h(\phi^n(x))}{d^n}
$$
and can be characterized by the the identity $\hhat_\phi(\phi(x))=d\hhat_\phi(x)$ together with the fact that $h-\hhat_\phi$ is bounded on $\Kbar$.  Locally, given a finite extension $L/K$, for each place $v\in M_L$ define the {\em canonical local height function} by
\begin{equation}\label{LocalHeightDef}
\lambda_{\phi,v}:\CC_v\to\RR \hskip1cm \lambda_{\phi,v}(x) = \lim_{n\to+\infty}\frac{1}{d^n}\log^+|\phi^n(x)|_v.
\end{equation}
Then an alternative expression for the canonical height is given by
\begin{equation}\label{CanonicalHeightDef}
\hhat_\phi(\alpha) = \sum_{v\in M_L}r_v\lambda_{\phi,v}(\alpha),
\end{equation}
for all $\alpha\in L$, a formula which may be viewed as analogous to (\ref{WeilHeightDef}).

For each place $v\in M_K$, standard arguments show that $\lambda_{\phi,v}(x)\geq0$ for all $x\in \CC_v$, with equality if and only if $x$ is in the filled Julia set 
\begin{equation*}\label{FilledJuliaDef}
F_{\phi,v} = \{x\in\CC_v \mid |\phi^n(x)|_v \text{ is bounded as } n\to+\infty\}
\end{equation*}
associated to $\phi$.  The  {\em canonical measure} $\mu_{\phi,v}$ associated to $\phi$ is a $\phi$-invariant unit Borel measure supported on $F_{\phi,v}$ which describes the limiting distribution of preperiodic points and iterated inverse images with respect to $\phi$.  There are several equivalent constructions of this measure in the literature; see \cite{MR736568}, \cite{MR741393} in the Archimedean case and \cite{MR2244226}, \cite{MR2244803}, \cite{MR2221116} in the non-Archimedean case.  (Technically, when $v$ is a non-Archimedean place, the objects $\lambda_{\phi,v}$, $F_{\phi,v}$, and $\mu_{\phi,v}$ need to be interpreted on the Berkovich affine line $\Abf_v^1$.  We will not need to go into these details in the present paper.)

Given two polynomials $\phi(x),\psi(x)\in K[x]$ of degree at least two, the {\em Arakelov-Zhang pairing} can be defined by either of the two expressions
\begin{equation}\label{AZPairing}
\langle\phi,\psi\rangle = \sum_{v\in M_K}r_v\int\lambda_{\phi,v}d\mu_{\psi,v} =\sum_{v\in M_K}r_v\int\lambda_{\psi,v}d\mu_{\phi,v}.
\end{equation}
Thus $\langle\phi,\psi\rangle$ is a nonnegative real number, and in some sense it  measures the global arithmetic-dynamical distance between the two maps.  This pairing was originally defined as a limit of arithmetic intersection products by Zhang \cite{MR1311351}, and described analytically using Berkovich spaces by Petsche-Szpiro-Tucker \cite{MR2869188}.  For our purposes the most important fact about the Arakelov-Zhang pairing is that it is closely related to points which have small canonical height with respect to one of the two maps.  In particular, it was shown in \cite{MR2869188} that if $\{\alpha_n\}$ is a sequence of distinct points in $\Kbar$ with $\hhat_{\phi}(\alpha_n)\to0$, then $\hhat_{\psi}(\alpha_n)\to \langle\phi,\psi\rangle$. 

In the special case $\psi(x)=x^d$ for $d\geq2$, the canonical height $\hhat_{\psi}$ is the same as the usual Weil height $h$, $\lambda_{\psi,v}(\cdot)=\log^+|\cdot|_v$, $F_{\psi,v}$ is the closed unit disc, and $\mu_{\psi,v}$ is equal to the normalized Haar measure supported on the unit circle of $\CC_v=\CC$ when $v$ is Archimedean, and equal to the Dirac measure supported at the Gauss point of $\Abf_v^1$ when $v$ is non-Archimedean.  In particular, the value of the pairing $\langle\phi,x^d\rangle$ does not depend on $d$.

\begin{thm}\label{SmallAZPairing}
Let $\phi(x)\in K[x]$ be a polynomial of degree $d\geq2$ defined over $K$ such that $\langle\phi,x^d\rangle>0$, and let $\alpha$ be a non-exceptional point for $\phi$.  If $K_\infty(\phi,\alpha)\subseteq L\subseteq \Kbar$, then
\begin{equation}\label{BogAZIneq}
B_0(L)\leq \langle\phi,x^d\rangle.
\end{equation}
\end{thm}
\begin{proof}
For each $n\geq1$, let $\alpha_n\in K_\infty\subseteq L$ satisfy $\phi^n(\alpha_n)=\alpha$; since $\alpha$ is not an exceptional point we may assume that the $\alpha_n$ are distinct.  It follows from properties of the canonical height that $\hhat_{\phi}(\alpha_n)=\hhat_{\phi}(\alpha)/d^n\to0$ as $n\to+\infty$.  By Theorem 1 of Petsche-Szpiro-Tucker \cite{MR2869188}, it follows that $h(\alpha_n)\to\langle\phi,x^d\rangle>0$, and (\ref{BogAZIneq}) follows from the definition of $B_0(L)$.
\end{proof}

As a sample application of Theorem~\ref{SmallAZPairing}, we can show that for any number field $K$, a certain infinite family of polynomials satisfies Conjecture~\ref{PolynomialConjecture}.

\begin{cor}\label{SmallAZPairingCor}
For each number field $K$, there exists a constant $C_K$ such that $\Gal(K_\infty(\frac{x^p-x}{p},\alpha)/K)$ is nonabelian over $K$ for all $\alpha\in K$ and all primes $p\geq C_K$.  In particular, $\Gal(K_\infty(\frac{x^p-x}{p},\alpha)/\QQ)$ is nonabelian for all $\alpha\in\QQ$ and all $p\geq29$.
\end{cor}
\begin{proof}
It has been shown by Petsche-Stacy \cite{MR3958063} that $\langle\frac{x^p-x}{p},x^d\rangle=\frac{\log p}{p-1}$.  Thus if $K_\infty(\frac{x^p-x}{p},\alpha)\subseteq K^\ab$, Theorem~\ref{SmallAZPairing} implies that $B_0(K^\ab)\leq \frac{\log p}{p-1}$.  But since $B_0(K^\ab)>0$ (Amoroso-Zannier \cite{MR1817715}), we have a contradiction for large enough $p$.  In particular, it was shown by Amoroso-Dvornicich \cite{MR1740514} that $B_0(\QQ^\ab)\geq\frac{\log 5}{12}$, which exceeds $\frac{\log p}{p-1}$ once $p\geq29$.
\end{proof}

We remark that, according to Conjecture~\ref{PolynomialConjecture}, we expect that $\Gal(K_\infty(\frac{x^p-x}{p},\alpha)/K)$ is nonabelian for all number fields $K$, all $\alpha\in K$, and all primes $p$.

\section{The map $x^2-1$}\label{x2MinusOneSect}

It is well known that there are exactly three $\QQ$-conjugacy classes of postcritically finite quadratic polynomials over $\QQ$, represented by $x^2$, $x^2-1$, and $x^2-2$.  By $\QQ$-conjugacy it suffices to check the family $\phi_c(x)=x^2+c$  for $c\in \QQ$, and the assumption that the critical point $0$ is preperiodic (i.e. $\phi_c^m(0)=\phi_c^n(0)$ for $m<n$) forces $c$ to be an algebraic integer (hence a rational integer) and also an element of the complex  Mandelbrot set $\Mcal=\{c\in\CC\mid\phi_c^n(0)\not\to\infty\}$.  It is elementary to check that $\Mcal\cap\ZZ=\{-2,-1,0\}$.

Since $x^2$ and $x^2-2$ are a powering map and a Chebyshev map, respectively, they are treated by Theorems~\ref{PoweringTheorem} and \ref{ChebyshevTheorem}, and the stable postcritically infinite quadratic case is treated in Theorem~\ref{StableNPCFCaseThm}.  Thus in order to complete the proof of Theorem~\ref{MainStabQuadThm}, it suffices to consider the polynomial $\phi(x)=x^2-1$ over $\QQ$ in the stable case.  

In order to show that $K_\infty(x^2-1,\alpha)/\QQ$ is never an abelian extension, one might hope to combine the bound $B_0(\QQ^\ab)\geq(\log5)/12=0.134...$ of Amoroso-Dvornicich with Theorem~\ref{SmallAZPairing}, but it turns out that the Arakelov-Zhang pairing $\langle x^2-1,x^2\rangle=0.167...$ is too large for this argument to apply directly.  However, we can recover this strategy (in the stable case) by showing that if $K_\infty(x^2-1,\alpha)$ is an abelian extension of $\QQ$ then it is contained in the subfield $\QQ(\mu_{2^\infty})$ of $\QQ^\ab$, which has Bogomolov constant $B_0( \QQ(\mu_{2^\infty}))\geq(\log2)/4=0.173...$, large enough to obtain a contradiction.

\begin{lem}\label{NoOddRamify}
Let $\phi(x)=x^2-1$, let $\alpha\in K$, and assume that the pair $(\phi,\alpha)$ is stable over $K$ and that $K_\infty=K_\infty(x^2-1,\alpha)$ is an abelian extension of $K$.  If $\pfrak$ is a prime of $K$ with residue characteristic not equal to $2$, then $\pfrak$ is unramified in $K_\infty$.
\end{lem}
\begin{proof}
The stability and abelian hypotheses imply via Lemma~\ref{ActionLem} that $[K_n:K]=2^n$ for all $n\geq1$.  Let $\psi(x)=\phi(x+\alpha)-\alpha=x^2+2\alpha x+\alpha^2-\alpha-1$.  The critical point of $\psi(x)$ is $c=-\alpha$, which is part of a $2$-cycle; that is, $\psi^n(c)=-\alpha$ for all even $n$, and $\psi^n(c)=-1-\alpha$ for all odd $n$.  Clearly
\begin{equation*}
\disc(\phi(x)-\alpha) = 4(1+\alpha)
\end{equation*}
and using Lemma~\ref{DiscIterateLem} for $n\geq2$ we have
\begin{equation*}
\begin{split}
\disc(\phi^n(x)-\alpha) & =\disc(\phi^n(x+\alpha)-\alpha) \\
& =\disc(\psi^n(x)) \\
& = 
\begin{cases}
R_n^2(-\alpha) & \text{ if } n\geq2 \text{ is even} \\
R_n^2(-1-\alpha) & \text{ if } n\geq3 \text{ is odd}
\end{cases}
\end{split}
\end{equation*}
for some nonzero $R_n\in K$.

\underline{Case 1:} $-1-\alpha$ is not a square in $K$.  Then $\disc(\phi^n(x)-\alpha)$ is not a square for all odd $n\geq3$, and thus viewing $\Gal(K_n/K)$ as a subgroup of $S_{2^n}$ via its action on the roots of $\phi^n(x)-\alpha$, $\Gal(K_n/K)$ is not a subgroup of $A_{2^n}$ for all odd $n\geq3$.  By Lemma~\ref{CyclicSnLem}, it follows that $\Gal(K_n/K)$ is cyclic for all odd $n\geq3$.  Then $\Gal(K_n/K)$ must be cyclic for all $n$, since $\Gal(K_m/K)$ is a quotient of $\Gal(K_n/K)$ when $1\leq m<n$.  We conclude from Lemma~\ref{CyclicTowerRamLem} that no primes of $K$ of odd residue characteristic ramify in $K_\infty$.

\underline{Case 2:} $-1-\alpha=t^2$ for $t\in K$.  By the stability hypothesis we know that $\disc(\phi(x)-\alpha)=4(1+\alpha)=-4t^2$ is not a square in $K$, and thus we conclude that $-1$ is not a square in $K$ and $K_1=K(\phi^{-1}(\alpha))=K(\sqrt{4(1+\alpha)})=K(\sqrt{-1})$.  In particular, no primes of $K$ with odd residue characteristic ramify in $K_1$, so it now suffices to show that no primes of $K_1$ with odd residue characteristic ramify in $K_\infty$.

Let $\phi^{-1}(\alpha)=\{\alpha',\alpha''\}$.  Thus for each $n\geq1$, we have a disjoint union
$$
\phi^{-n}(\alpha)=\phi^{-(n-1)}(\alpha')\amalg\phi^{-(n-1)}(\alpha'').
$$
and it follows from the transitive action of $\Gal(K_n/K)$ on $\phi^{-n}(\alpha)$ that $\Gal(K_n/K_1)$ acts transitively on $\phi^{-(n-1)}(\alpha')$.  We conclude that $(\phi,\alpha')$ is a stable pair over $K_1$.  Arguing as above, it follows that $1+\alpha'$ is not a square in $K_1$, and as $-1$ is a square in $K_1$, we deduce that $-1-\alpha'$ is not a square in $K_1$.  We are now in the setting of Case 1 for the pair $(\phi,\alpha')$ over $K_1$, and we conclude that no primes of $K_1$ with odd residue characteristic ramify in $K_\infty$.
\end{proof}

\begin{thm}
Let $\phi(x)=x^2-1$, let $\alpha\in \QQ$, and assume that the pair $(\phi,\alpha)$ is stable.  Then $\Gal(K_\infty/\QQ)$ is nonabelian.
\end{thm}
\begin{proof}
Assume on the contrary that $K_\infty/\QQ$ is an abelian extension.  By Lemma~\ref{NoOddRamify}, no odd primes ramify in $K_\infty$, and it follows from class field theory that $K_\infty\subseteq  \QQ(\mu_{2^\infty})$.  As we will calculate in $\S$~\ref{ApproxSect}, we have
\begin{equation}\label{AZPairingBasilica}
\langle x^2-1,x^2\rangle = 0.167...,
\end{equation}
and we conclude from Theorem~\ref{SmallAZPairing} that $B_0( \QQ(\mu_{2^\infty}))\leq \langle x^2-1,x^2\rangle=0.167...$.  This is a contradiction of the bound $B_0( \QQ(\mu_{2^\infty}))\geq(\log2)/4=0.173...$ proved in Theorem~\ref{AD2Theorem}.
\end{proof}

\section{Numerical approximation of \\ Archimedean Arakelov-Zhang integrals}\label{ApproxSect}

In this section we prove a result which may be of use in numerically approximating the value of the Archimedean part of the Arakelov-Zhang pairing $\langle x^2+c,x^2\rangle$.

Given a polynomial $\phi(x)\in\CC[x]$ of degree $2$, we may express the corresponding canonical local height function as
\begin{equation}\label{LocalHeightL2Norm}
\lambda_{\phi}(x)=\lim_{n\to+\infty}\frac{1}{2^n}\log(1+|\phi^n(x)|^2)^{1/2}.
\end{equation}
Note that, compared to (\ref{LocalHeightDef}), we have replaced $\log^+|\cdot|$ with $\log(1+|\cdot|^2)^{1/2}$; this choice gives differentiable approximations to $\lambda_\phi(x)$, but in the limit it defines precisely the same Archimedean local height function as in (\ref{LocalHeightDef}).

Next define
\begin{equation}\label{Const}
B(\phi)=\sup_{x\in\CC}|\log(1+|\phi(x)|^2)^{1/2}-2\log(1+|x|^2)^{1/2}|.
\end{equation}
The significance of this constant follows from a standard telescoping series argument (\cite{MR2316407} Thm. 3.20), which shows that for each $n\geq1$ we have
\begin{equation}\label{TelescopingIneq}
\begin{split}
\bigg|\lambda_{\phi}(x)-\frac{1}{2^{n}}\log(1+|\phi^n(x)|^2)^{1/2}\bigg| \leq \frac{B(\phi)}{2^n}.
\end{split}
\end{equation}

\begin{prop}\label{AZApproxProp}
Let $\phi(x)=x^2+c$ for $c\in\CC$, let $N\geq1$ and $M\geq1$ be integers, and let $\mu_M$ denote the set of $M$-th roots of unity in $\CC$.  Then 
\begin{equation}\label{AZNumericalBound}
\bigg|\int_{0}^{1}\lambda_{\phi}(e^{2\pi it})dt-\frac{1}{M}\sum_{\zeta\in \mu_M}\frac{1}{2^N}\log(1+|\phi^N(\zeta)|^2)^{1/2}\bigg|\leq \frac{B(\phi)}{2^N}+\frac{\pi T^N}{M}
\end{equation}
where $T=(1+\sqrt{1+4|c|})/2$.
\end{prop}
\begin{proof}
Using (\ref{TelescopingIneq}) we obtain
\begin{equation}\label{FirstEstimate}
\bigg|\int_{0}^{1}\lambda_{\phi}(e^{2\pi it})dt - \int_{0}^{1}f_N(t)dt\bigg| \leq \frac{B(\phi)}{2^N}
\end{equation}
where
$$
f_N:\RR/\ZZ\to\RR \hskip1cm f_N(t)=\frac{1}{2^{N}}\log(1+|\phi^N(e^{2\pi it})|^2)^{1/2}.
$$

If $f:\RR/\ZZ\to\RR$ is continuously differentiable and $M\geq1$, then 
\begin{equation}\label{RiemannSumEst}
\bigg|\frac{1}{M}\sum_{m=1}^{M}f(m/M)-\int_0^1f(t)dt\bigg|\leq \frac{\|f'\|_\infty}{2M},
\end{equation}
where $\|g\|_\infty=\sup\{|g(t)|\mid t\in\RR/\ZZ\}$.  This inequality can be verified using an elementary Riemann sum argument, together with the mean-value theorem estimate $|f(\frac{m}{M})-f(t)|\leq\|f'\|_\infty/2M$ whenever $t\in[\frac{m}{M}-\frac{1}{2M},\frac{m}{M}+\frac{1}{2M}]$.

We are going to show that $\|f_N'\|_\infty\leq 2\pi T^N$, and so applying (\ref{RiemannSumEst}) with $f=f_N$, and combining with (\ref{FirstEstimate}), we obtain (\ref{AZNumericalBound}).

It remains only to prove $\|f_N'\|_\infty\leq 2\pi T^N$.  Given any polynomial $F(x)\in\CC[x]$, an application of the multivariable chain rule gives, for $t\in\RR$, 
\begin{equation}\label{DerivativeFN}
\bigg|\frac{d}{dt}\log(1+|F(e^{2\pi it})|^2)^{1/2}\bigg| \leq \frac{2\pi|F(e^{2\pi it})||F'(e^{2\pi it})|}{1+|F(e^{2\pi it})|^2}.
\end{equation}

The choice of $T$ in terms of $c$ was made so that $T^2-T-|c|=0$, which can be used to show that 
\begin{equation}\label{ChoiceOfTImplies}
|x|>T \hskip1cm \Rightarrow \hskip1cm |\phi(x)| > |x|^2/T > T
\end{equation}
for all $x\in\CC$.  Indeed, $|\phi(x)|\geq |x|^2-|c|$ and therefore
$$
\frac{|\phi(x)|}{|x|^2}\geq1-\frac{|c|}{|x|^2}>1-\frac{|c|}{T^2}=\frac{1}{T}.
$$

Using an iteration of the inequality (\ref{ChoiceOfTImplies}), we claim that for each $x\in \CC$, we have
\begin{equation}\label{OrbitProductBound}
|x\phi(x)\phi^2(x)\dots\phi^{N}(x)|\leq T^N(1+|\phi^N(x)|^2).
\end{equation}
Indeed, if $|\phi^n(x)|\leq T$ for all $n=0,1,2,\dots,N-1$, the bound is trivial.  Otherwise, let $0\leq n_0\leq N-1$ be the smallest $n$ for which $|\phi^n(x)|>T$.  Using (\ref{ChoiceOfTImplies}) we have $|\phi^n(x)|<T^{1/2}|\phi^{n+1}(x)|^{1/2}$ for all $n_0\leq n\leq N-1$.  Thus, letting $m_0=N-n_0$, we have
\begin{equation*}
\begin{split}
|x\phi(x)\phi^2(x)\dots\phi^{N}(x)| 
	& \leq T^{n_0}|\phi^{n_0}(x)||\phi^{n_0+1}(x)|\dots|\phi^{N}(x)|  \\
	& < T^{n_0}T^{\frac{1}{2}}|\phi^{n_0+1}(x)|^{\frac{1}{2}+1}|\phi^{n_0+2}(x)|\dots|\phi^{N}(x)|  \\
	& < T^{n_0}T^{\frac{1}{2}}T^{\frac{1}{4}+\frac{1}{2}}|\phi^{n_0+2}(x)|^{\frac{1}{4}+\frac{1}{2}+1}|\phi^{n_0+3}(x)|\dots|\phi^{N}(x)|  \\	
	& \vdots \\
	& < T^{n_0}T^{\frac{1}{2}}T^{\frac{1}{4}+\frac{1}{2}}\dots T^{\frac{1}{2^{m_0}}+\dots+\frac{1}{2}}|\phi^{N}(x)|^{\frac{1}{2^{m_0}}+\dots+\frac{1}{2}+1}  \\
	& < T^N(1+|\phi^N(x)|^2).
\end{split}
\end{equation*}
Using $\phi'(x)=2x$ and the chain rule we have
\begin{equation}
\begin{split}
(\phi^N)'(x) & = \phi'(\phi^{N-1}(x))\phi'(\phi^{N-2}(x))\dots \phi'(\phi(x))\phi'(x) \\
	& = 2^N\phi^{N-1}(x)\phi^{N-2}(x)\dots \phi(x)x,
\end{split}
\end{equation}
and applying (\ref{OrbitProductBound}) we obtain
\begin{equation}\label{ProductBound}
\begin{split}
|\phi^N(x)||(\phi^N)'(x)| & = 2^N|\phi^{N}(x)\phi^{N-1}(x)\phi^{N-2}(x)\dots \phi(x)x| \\
	& \leq 2^NT^N(1+|\phi^N(x)|^2).
\end{split}
\end{equation}
Finally, taking $F(x)=\phi^N(x)$ in (\ref{DerivativeFN}) and using the bound (\ref{ProductBound}) we conclude $\|f_N'\|_\infty\leq 2\pi T^N$, completing the proof.
\end{proof}

We conclude with an explanation of how to use Proposition~\ref{AZApproxProp} to obtain the numerical calculation (\ref{AZPairingBasilica}) of $\langle x^2-1,x^2\rangle$ to the specified precision.  Since both $\phi(x)=x^2-1$ and $\psi(x)=x^2$ are monic with integer coefficients, the non-Archimedean contributions in (\ref{AZPairing}) vanish, and since the Archimedean canonical measure $\mu_{\phi,\infty}$ is the normalized Haar measure supported on the unit circle of $\CC$, we have  
\begin{equation}\label{AZPairingExample}
\langle \phi, \psi \rangle = \int_{0}^{1}\lambda_{\phi,\infty}(e^{2\pi it})dt.
\end{equation}

We approximate this integral using Proposition~\ref{AZApproxProp} with $\phi(x)=x^2-1$.  Clearly $T=(1+\sqrt{5})/2$, and we will show that $B(\phi)=(\log5)/2$, which according to the definition (\ref{Const}) is equivalent to checking that
\begin{equation}\label{ConstExample}
\sup_{x\in\CC}\left|\log\frac{1+|x^2-1|^2}{(1+|x|^2)^2}\right|=\log5.
\end{equation}
To establish (\ref{ConstExample}) is to prove both of the inequalities
\begin{equation}\label{ConstExample2}
\textstyle\frac{1}{5}(1+|x|^2)^2 \leq 1+|x^2-1|^2 \leq 5(1+|x|^2)^2
\end{equation}
and to prove that at least one of them is sharp.  The second inequality (with the stronger constant $2$ in place of $5$) is easily checked using only the triangle inequality.  For the second inequality, we see using the triangle inequality that $1+|x^2-1|^2\geq|x|^4-2|x|^2+2$, with equality when $x$ is real; we complete the proof by noting that $\frac{r^4-2r^2+2}{(1+r^2)^2}$ is minimized at $r=\sqrt{3/2}$ with minumum value $1/5$.

Taking $N=13$ and $M=2^{24}$, we have
\begin{equation}\label{AZPairingExampleNumerical}
\int_{0}^{1}\lambda_{\phi,\infty}(e^{2\pi it})dt = \frac{1}{2^{38}}\sum_{\zeta\in \mu_{2^{24}}}\log(1+|\phi^{13}(\zeta)|^2)+\theta
\end{equation}
where $|\theta|\leq (\log5)/2^{14}+ \pi ((1+\sqrt{5})/2)^{13}/2^{24}=0.000195...$.  Finally, one may perform the calculation
\begin{equation}\label{AZPairingExampleNumericalSum}
\frac{1}{2^{38}}\sum_{\zeta\in \mu_{2^{24}}}\log(1+|\phi^{13}(\zeta)|^2) = 0.16772223...
\end{equation}
using any implementation of arbitrary precision floating-point arithmetic; specifically we used the RealField() package in SageMath \cite{sagemath} in a computation taking about three hours.  It follows from (\ref{AZPairingExample}), (\ref{AZPairingExampleNumerical}), (\ref{AZPairingExampleNumericalSum}), and the bound on $\theta$ that (\ref{AZPairingBasilica}) is accurate to the indicated precision.

It was pointed out to us by an anonymous referee that, if one only wants to check that $\langle x^2-1,x^2\rangle<(\log2)/4$ but without giving an explicit numerical approximation for $\langle x^2-1,x^2\rangle$, then one could get away with the smaller parameters $N=9$ and $M=2^{16}$, resulting in a much faster calculation.



\begin{thebibliography}{10}

\bibitem{ABCCF}
{\sc F.~Ahmad, R.~Benedetto, J.~Cain, G.~Carroll, and L.~Fang}, {\em The
  arithmetic basilica: a quadratic {PCF} arboreal {G}alois group}, arxiv:
  1909.00039,  (2019).

\bibitem{MR1740514}
{\sc F.~Amoroso and R.~Dvornicich}, {\em A lower bound for the height in
  abelian extensions}, J. Number Theory, 80 (2000), pp.~260--272.

\bibitem{MR1817715}
{\sc F.~Amoroso and U.~Zannier}, {\em A relative {D}obrowolski lower bound over
  abelian extensions}, Ann. Scuola Norm. Sup. Pisa Cl. Sci. (4), 29 (2000),
  pp.~711--727.

\bibitem{MR2244226}
{\sc M.~H. Baker and R.~Rumely}, {\em Equidistribution of small points,
  rational dynamics, and potential theory}, Ann. Inst. Fourier (Grenoble), 56
  (2006), pp.~625--688.

\bibitem{bombierigubler}
{\sc E.~Bombieri and W.~Gubler}, {\em Heights in {D}iophantine Geometry}, no.~4
  in New Mathematical Monographs, Cambridge University Press, Cambridge, 2006.

\bibitem{BostonTreeRep}
{\sc N.~Boston}, {\em Tree representations of {G}alois groups}, preprint,
  https://arxiv.org/abs/math/0009256.

\bibitem{MR2318536}
{\sc N.~Boston and R.~Jones}, {\em Arboreal {G}alois representations}, Geom.
  Dedicata, 124 (2007), pp.~27--35.

\bibitem{MR2520459}
\leavevmode\vrule height 2pt depth -1.6pt width 23pt, {\em The image of an
  arboreal {G}alois representation}, Pure Appl. Math. Q., 5 (2009),
  pp.~213--225.

\bibitem{MR3762687}
{\sc A.~Bridy, P.~Ingram, R.~Jones, J.~Juul, A.~Levy, M.~Manes,
  S.~Rubinstein-Salzedo, and J.~H. Silverman}, {\em Finite ramification for
  preimage fields of post-critically finite morphisms}, Math. Res. Lett., 24
  (2017), pp.~1633--1647.

\bibitem{MR911121}
{\sc J.~W.~S. Cassels and A.~Fr{\"o}hlich}, eds., {\em Algebraic {N}umber
  {T}heory}, London, 1986, Academic Press Inc. [Harcourt Brace Jovanovich
  Publishers].
\newblock Reprint of the 1967 original.

\bibitem{MR2244803}
{\sc A.~Chambert-Loir}, {\em Mesures et \'equidistribution sur les espaces de
  {B}erkovich}, J. Reine Angew. Math., 595 (2006), pp.~215--235.

\bibitem{MR2221116}
{\sc C.~Favre and J.~Rivera-Letelier}, {\em \'{E}quidistribution quantitative
  des points de petite hauteur sur la droite projective}, Math. Ann., 335
  (2006), pp.~311--361.

\bibitem{FerragutiPagano}
{\sc A.~Ferraguti and C.~Pagano}, {\em Constraining images of quadratic
  arboreal representations}, arxiv: 2004.02847,  (2020).

\bibitem{MR736568}
{\sc A.~Freire, A.~Lopes, and R.~Ma\~{n}\'{e}}, {\em An invariant measure for
  rational maps}, Bol. Soc. Brasil. Mat., 14 (1983), pp.~45--62.

\bibitem{jones:thesis}
{\sc R.~Jones}, {\em {G}alois {M}artingales and the {$p$}-adic {H}yperbolic
  {M}andelbrot {S}et}, PhD thesis, Brown University, 2005.

\bibitem{jones:itgaltow}
\leavevmode\vrule height 2pt depth -1.6pt width 23pt, {\em Iterated {G}alois
  towers, their associated martingales, and the {$p$}-adic {M}andelbrot set},
  Compos. Math., 143 (2007), pp.~1108--1126.

\bibitem{jones:denprimdiv}
\leavevmode\vrule height 2pt depth -1.6pt width 23pt, {\em The density of prime
  divisors in the arithmetic dynamics of quadratic polynomials}, J. Lond. Math.
  Soc. (2), 78 (2008), pp.~523--544.

\bibitem{MR3220023}
\leavevmode\vrule height 2pt depth -1.6pt width 23pt, {\em Galois
  representations from pre-image trees: an arboreal survey}, in Actes de la
  {C}onf\'erence ``{T}h\'eorie des {N}ombres et {A}pplications'', vol.~2013 of
  Publ. Math. Besan{c}on Alg\`ebre Th\'eorie Nr., Presses Univ.
  Franche-Comt\'e, Besan{c}on, 2013, pp.~107--136.

\bibitem{lang:numbertheory}
{\sc S.~Lang}, {\em Algebraic {N}umber {T}heory}, vol.~110 of Graduate Texts in
  Mathematics, Springer-Verlag, New York, second~ed., 1994.

\bibitem{MR741393}
{\sc M.~J. Ljubich}, {\em Entropy properties of rational endomorphisms of the
  {R}iemann sphere}, Ergodic Theory Dynam. Systems, 3 (1983), pp.~351--385.

\bibitem{MR3937588}
{\sc N.~Looper}, {\em Dynamical {G}alois groups of trinomials and {O}doni's
  conjecture}, Bull. Lond. Math. Soc., 51 (2019), pp.~278--292.

\bibitem{MR805714}
{\sc R.~W.~K. Odoni}, {\em The {G}alois theory of iterates and composites of
  polynomials}, Proc. London Math. Soc. (3), 51 (1985), pp.~385--414.

\bibitem{MR813379}
\leavevmode\vrule height 2pt depth -1.6pt width 23pt, {\em On the prime
  divisors of the sequence {$w\sb {n+1}=1+w\sb 1\cdots w\sb n$}}, J. London
  Math. Soc. (2), 32 (1985), pp.~1--11.

\bibitem{MR962740}
\leavevmode\vrule height 2pt depth -1.6pt width 23pt, {\em Realising wreath
  products of cyclic groups as {G}alois groups}, Mathematika, 35 (1988),
  pp.~101--113.

\bibitem{MR1418355}
\leavevmode\vrule height 2pt depth -1.6pt width 23pt, {\em On the {G}alois
  groups of iterated generic additive polynomials}, Math. Proc. Cambridge
  Philos. Soc., 121 (1997), pp.~1--6.

\bibitem{MR3958063}
{\sc C.~Petsche and E.~Stacy}, {\em A dynamical construction of small totally
  {$p$}-adic algebraic numbers}, J. Number Theory, 202 (2019), pp.~27--36.

\bibitem{MR2869188}
{\sc C.~Petsche, L.~Szpiro, and T.~J. Tucker}, {\em A dynamical pairing between
  two rational maps}, Trans. Amer. Math. Soc., 364 (2012), pp.~1687--1710.

\bibitem{MR2316407}
{\sc J.~H. Silverman}, {\em The {A}rithmetic of {D}ynamical {S}ystems},
  vol.~241 of Graduate Texts in Mathematics, Springer, New York, 2007.

\bibitem{Specter}
{\sc J.~Specter}, {\em Polynomials with surjective arboreal {G}alois
  representations exist in every degree}, 2018.
\newblock \url{arxiv.org/abs/1803.00434}.

\bibitem{MR1174401}
{\sc M.~Stoll}, {\em Galois groups over {$\mathbb{Q}$} of some iterated
  polynomials}, Arch. Math. (Basel), 59 (1992), pp.~239--244.

\bibitem{sagemath}
{\sc {The Sage Developers}}, {\em {S}ageMath, the {S}age {M}athematics
  {S}oftware {S}ystem ({V}ersion 7.1)}, 2016.
\newblock {\tt https://www.sagemath.org}.

\bibitem{MR1311351}
{\sc S.~Zhang}, {\em Small points and adelic metrics}, J. Algebraic Geom., 4
  (1995), pp.~281--300.

\end{thebibliography}

\end{document}